\documentclass[a4paper,12pt]{amsart}

\usepackage{amsmath}
\usepackage{amssymb}
\usepackage{mathrsfs}
\usepackage{enumerate}
\usepackage{ifthen}
\usepackage{graphicx}
\usepackage[T1]{fontenc} 

\nonstopmode \numberwithin{equation}{section}
\newtheorem{thm}{Theorem}[section]
\newtheorem{cor}{Corollary}[section]
\newtheorem{lem}{Lemma}[section]

\theoremstyle{definition}

\newtheorem{example}{Example}[section]

\newcounter{minutes}
\divide\time by 60
\newcounter{hours}
\multiply\time by 60
\addtocounter{minutes}{-\time}

\newcounter {own}
\def\theown {\thesection       .\arabic{own}}

{\qed\bigskip}

\newcounter{alphabet}

\begin{document}

\title{ON THE PRE-SCHWARZIAN NORM OF CERTAIN LOGHARMONIC MAPPINGS}

\author{Md Firoz Ali}
\address{Md Firoz Ali,
Department of Mathematics,
National Institute of Technology Durgapur,
West Bengal - 713209, India.}
\email{ali.firoz89@gmail.com, firoz.ali@maths.nitdgp.ac.in}

\author{Sushil Pandit}
\address{Sushil Pandit,
Department of Mathematics,
National Institute of Technology Durgapur,
West Bengal - 713209, India.}
\email{sushilpandit15594@gmail.com}


\def\thefootnote{}
\footnotetext{ {\tiny File:~\jobname.tex,
printed: \number\year-\number\month-\number\day,
          \thehours.\ifnum\theminutes<10{0}\fi\theminutes }
} \makeatletter\def\thefootnote{\@arabic\c@footnote}\makeatother

\begin{abstract}
We connect the pre-Schwarzian norm of logharmonic mappings to the pre-Schwarzian norm of an analytic function and establish some necessary and sufficient conditions under which locally univalent logharmonic mappings have a finite pre-Schwarzian norm. We also obtain a necessary and sufficient condition for a logharmonic  function to be Bloch. Furthermore, we obtain the pre-Schwarzian norm and growth theorem for logharmonic Bloch mappings and their analytic and co-analytic parts.
\end{abstract}

\thanks{}

\maketitle
\pagestyle{myheadings}
\markboth{Md Firoz Ali, Sushil Pandit}{Pre-Schwarzian norm and Schwarzian norm}

\section{Introduction}
The pre-Schwarzian and Schwarzian derivative of locally univalent analytic mappings has become a widely used technique in the study of the geometric properties of such mappings. For instance, it can be used to get either necessary or sufficient conditions for the global univalence, or to obtain certain geometric conditions on the range of the corresponding functions involved there. Due to such beauty of these derivatives, the theory has been extended to complex valued harmonic mappings (see \cite{Chuaqui-Duren-Osgood-2003,Hernandez-Martin-2015}). Now, it is a natural question whether these derivatives can be defined for a logharmonic mapping. In this connection, in 2018, Liu and Ponnusamy \cite{Liu-Ponnusamy-2018} and in 2022, Bravo et al. \cite{Bravo-Hernandez-Ponnusamy-Venegas-2022} proposed  a definition of pre-Schwarzian and Schwarzian derivative of locally univalent logharmonic mapping. In this article, we study the pre-Schwarzian norm of logharmonic mappings and logharmonic Bloch mappings defined in the unit disk. \\

Let $\mathcal{S}_0$ denotes the class of all univalent analytic function $h$ in the unit disk $\mathbb{D}= \{z\in\mathbb{C}:|z|<1\}$ and $\mathcal{S}_1$ be the class of members $h\in\mathcal{S}_0$ with normalization $h'(0)=1.$ In a simply connected domain $\Omega,$ every complex valued harmonic mapping $f$ has a canonical representation of the form $f = h+\overline{g},$ where $h$ and $g$ are analytic functions in  $\Omega,$ called the analytic and co-analytic part of $f,$ respectively. The Jacobian of $f=h+\overline{g}$  is defined by $ J_f(z)=|f_z|^2 - |f_{\overline{z}}|^2=|h'(z)|^2-|g'(z)|^2.$ A logharmonic mapping $f$ is a solution of the nonlinear elliptic partial differential equation
\begin{align}\label{p2-001}
\frac{\overline{f_{\overline{z}}}}{\overline{f}}=\omega\frac{f_z}{f}
\end{align}
where the second dilatation function $\omega:\mathbb{D}\rightarrow\mathbb{D}$ is analytic and the Jacobian of $f$ is given by $J_f=|f_z|^2-|f_{\overline{z}}|^2=|f_z|^2(1-|\omega|^2).$ Note that $J_f$ is always positive, and so $f$ is a sense-preserving logharmonic mapping. If $f$ is non-constant and vanishes only at $z = 0,$ then $f$ admits the following representation
\begin{align*}
f(z) = z^m|z|^{2\beta m}h(z)\overline{g(z)},
\end{align*}
where $m$ is non-negative integer, ${\rm Re\,}(\beta)>-1/2$ and $h$ and $g$ are analytic functions in $\mathbb{D}$ such that $g(0)=1$ and $h(0)\neq 0$ (see \cite{Abdulhadi-1988,Abdulhadi-Ali-2012}). On the other hand, It has been shown that \cite{Abdulhadi-1988} if $f$ is a nonvanishing logharmonic mapping, then $f$ can be expressed as
\begin{align}\label{p2-005}
f(z) = h(z)\overline{g(z)},
\end{align}
where $h$ and $g$ are nonvanishing analytic functions in the unit disk $\mathbb{D}.$ Further, if the mapping $f$ given by \eqref{p2-005} is locally univalent and sense-preserving, then $h'g \neq 0$ in $\mathbb{D}$ and the second complex dilatation $\omega$ is given by
\begin{align}\label{p2-010}
\omega=g'h/gh'.
\end{align}

It is easy to see that if $f=h\overline{g}$ is a nonvanishing logharmonic mapping defined in $\mathbb{D}$, then $\log{f}=
\log{h}+\overline{\log{g}}$ is a harmonic mapping defined in $\mathbb{D},$ and its corresponding dilatations are the same.
As in \cite{Bravo-Hernandez-Ponnusamy-Venegas-2022}, we consider the logharmonic mappings of the form \eqref{p2-005} in the unit disk $\mathbb{D}$ such that $h$ is an analytic and locally univalent mapping, not necessarily different from zero, and $g$ is nonvanishing in $\mathbb{D}$ which happen when $f$ is locally univalent. There are several fundamental results on logharmonic mappings defined on the unit disk $\mathbb{D}$ (see \cite{Abdulhadi-Ali-2012,Liu-Ponnusamy-2018,Mao-Ponnusamy-Wang-2013}). \\

The classical Bloch theorem asserts the existence of a positive constant $m$ such that for any holomorphic mapping $h$ in the unit disk $\mathbb{D}$, the image $h(\mathbb{D})$ contains a Schlicht disk of radius $m$. By Schlicht disk, we mean a disk that is the univalent image of some region in $\mathbb{D}$. The Bloch constant is defined as the supremum of such constants $m.$ An analytic function $h$  defined in $\mathbb{D}$ is called a Bloch function (\cite{Anderson-Clunie-Pommerenke-1974, Pommerenke-1970}) if
\begin{align*}
\beta_h=\sup_{z\in\mathbb{D}}(1-|z|^2)|h'(z)|<\infty.
\end{align*}
The space $\mathcal{B}$ of analytic Bloch functions in the unit disk $\mathbb{D}$ forms a Banach space under the norm given by
\begin{align*}
||h||_{\mathcal{B}}=|h(0)|+\sup_{z\in\mathbb{D}}(1-|z|^2)|h'(z)|.
\end{align*}
A nonvanishing logharmonic mapping $f(z) = h(z)\overline{g(z)}$ in $\mathbb{D}$ is said to be a logharmonic Bloch function if
\begin{align}\label{p2-015}
\beta_f=\sup_{z\in\mathbb{D}}(1-|z|^2)\left\{\left|\frac{h'(z)}{h(z)}\right|+\left|\frac{g'(z)}{g(z)}\right|\right\}<\infty,
\end{align}
where $h$ and $g$ are analytic in $\mathbb{D}.$ Let $\mathcal{B}_{Lh}$ denote the space of all logharmonic Bloch functions.
The space $\mathcal{B}_{Lh}$ forms a complex Banach space with the norm given by (see \cite{Liu-Ponnusamy-2018})
\begin{align*}
||f||_{\mathcal{B}_{Lh}}=|f(0)|+\sup_{z\in\mathbb{D}}(1-|z|^2)\left\{\left|\frac{h'(z)}{h(z)}\right|+\left|\frac{g'(z)}{g(z)}\right|\right\}.
\end{align*}
This norm is known as the logharmonic Bloch norm and elements of this space are known as logharmonic Bloch functions. Liu and Ponnusamy \cite{Liu-Ponnusamy-2018} proved that the space $\mathcal{B}_{Lh}$ is linear and affine invariant. Indeed, they proved that if $f=h\overline{g}\in\mathcal{B}_{Lh}$ then $f^a\overline{f}^b\in\mathcal{B}_{Lh}$ for all $a, b\in\mathbb{C}$ (affine invariance) and $f\circ\phi\in\mathcal{B}_{Lh}$ for all automorphism $\phi(z)=(z-\alpha)/(1-\overline{\alpha}z),~\alpha\in\mathbb{D}$ (linear invariance).\\


\section{Pre-Schwarzian and Schwarzian norm}
For a locally univalent analytic function $h$ defined in a simply connected domain $\Omega$, the pre-Schwarzian derivative $P_h$ and the Schwarzian derivative $S_h$ are defined as
\begin{align}\label{p2-020}
P_h(z)=\frac{h''(z)}{h'(z)}\quad\text{and}\quad S_h(z) = P_h'(z)-\frac{1}{2}P_h^2(z)=\frac{h'''(z)}{h''(z)}-\frac{3}{2}\left(\frac{h''(z)}{h'(z)}\right)^2,
\end{align}
respectively. Moreover, the pre-Schwarzian norm and the Schwarzian norm of $h$ are defined by
\begin{align}\label{p2-025}
||P_h|| = \sup_{z \in \mathbb{D}}(1-|z|^2)|P_h(z)|\quad\text{and}\quad ||S_h|| = \sup_{z \in \mathbb{D}}(1-|z|^2)^2|S_h(z)|,
\end{align}
respectively. Several important global univalence criteria for a locally univalent analytic function $h$ were obtained using the notions of pre-Schwarzian and Schwarzian derivatives of $h$. For a univalent function $h$, it is well known that $||P_h||\leq 6$ and $||S_h||\leq 6$ (see \cite{Kruas-1932}) and these estimates are sharp. On the other hand, for a locally univalent function $h$ it is known that if $||P_h||\leq 1$ (see \cite{Becker-1972}, \cite{Becker-Pommerenke-1984}) or $||S_h||\leq 2$ (see \cite{Nehari-1949}), then the function $h$ is univalent in $\mathbb{D}$. In 1976, Yamashita \cite{Yamashita-1976} proved that $||P_h||$ is finite if and only if $h$ is uniformly locally univalent in $\mathbb{D},$ that is, there exists a constant $\rho>0$ such that $h$ is univalent on the hyperbolic disk $|(z-a)/(1-\overline{a}z)|<\tanh\rho$ of radius $\rho$ for every $a\in\mathbb{D}.$\\


In 2015, Hern{\'a}ndez and Mart{\'i}n \cite{Hernandez-Martin-2015} defined the Schwarzian derivative of a locally univalent harmonic mapping $f=h+\overline{g}$ by\\
\begin{equation}\label{p2-030}
\begin{split}
S_f &= \left(\log J_f\right)_{zz}-\frac{1}{2}\left(\log J_f\right)_z^2\\
&= S_h+\frac{\overline{\omega}}{1-|\omega|^2}\left(\frac{h''}{h'}\omega'-\omega''\right)-\frac{3}{2}\left(\frac{\omega'\overline{\omega}}{1-|\omega|^2}\right)^2,
\end{split}
\end{equation}\\
where $S_h$ is the classical Schwarzian derivative of the analytic function $h$, $J_f$ is the Jacobian and $\omega$ is the dilatation of $f.$ The pre-Schwarzian derivative of $f=h+\overline{g}$ is defined as
\begin{align}\label{p2-035}
P_f = \left(\log(J_f)\right)_z = \frac{h''}{h'}-\frac{\overline{\omega}\omega'}{1-|\omega|^2}.
\end{align}
This notion of pre-Schwarzian and Schwarzian derivatives of harmonic functions is a generalization of the classical pre-Schwarzian and Schwarzian derivatives of analytic functions. Note that when $f$ is analytic, we have $\omega=0.$ It is also easy to see that $S_f = (P_f)_z-\frac{1}{2}(P_f)^2.$ As in the case of analytic functions, for a sense-preserving locally univalent harmonic mapping $f = h+\bar{g}$ in the unit disk $\mathbb{D}$, the pre-Schwarzian norm $||P_f||$ and the Schwarzian norm $||S_f||$ are defined by \eqref{p2-025}. Hern{\'a}ndez and Mart{\'i}n \cite{Hernandez-Martin-2015} proved that a sense-preserving harmonic mapping is uniformly locally univalent if and only if its Schwarzian norm is finite. Later, Liu and Ponnusamy \cite{Liu-Ponnusamy--2018} proved that a sense-preserving harmonic mapping is uniformly locally univalent if and only if its pre-Schwarzian norm is finite. \\

For a locally univalent logharmonic mapping $f$ of the form \eqref{p2-005}, Bravo et al. \cite{Bravo-Hernandez-Ponnusamy-Venegas-2022} defined the pre-Schwarzian derivative $P_f$ as
\begin{align}\label{p2-040}
P_f=\left(\log(J_f)\right)_z =\frac{h''}{h'}+\frac{g'}{g}-\frac{\overline{\omega}\omega'}{1-|\omega|^2}.
\end{align}
where $J_f$ is the  Jacobian and $\omega$ is the dilatation of the function $f.$ Further, the pre-Schwarzian norm $||P_f||$ is defined by \eqref{p2-025}. If $f$ is a sense-preserving logharmonic mapping of the form \eqref{p2-005} and $\phi$ is a locally univalent analytic function for which the composition $f\circ\phi$ is well defined, then the function $f\circ\phi$ is again a sense-preserving logharmonic mapping and the pre-Schwarzian derivative of it is given by
\begin{align*}
P_{f\circ\phi}=P_f(\phi)\phi'+P_\phi.
\end{align*}
For more information about the properties of the pre-Schwarzian derivative of a sense-preserving logharmonic mapping, we refer to \cite{Bravo-Hernandez-Ponnusamy-Venegas-2022}.\\

The theory of logharmonic mapping has become an exciting field of research in the last few years. Our primary goal of this article is to establish necessary and sufficient conditions under which a locally univalent logharmonic mapping has a finite pre-Schwarzian norm.  We also obtain a necessary and sufficient condition for a logharmonic  function to be Bloch. Furthermore, we obtain the pre-Schwarzian norm and growth theorem for logharmonic Bloch mappings and their analytic and co-analytic parts.


\section{Main Results}\label{background And Main Results}
Recently, Liu and Ponnusamy \cite{Liu-Ponnusamy--2018} have established a connection between the pre-Schwarzian norm of a sense-preserving harmonic mapping $f=h+\overline{g}$ and that of the analytic part $h$ and showed that the finiteness of their pre-Schwarzin norms behaves alike. We present a similar result for sense-preserving logharmonic mapping, which, in certain circumstances,  provides a necessary and sufficient condition for a logharmonic function to have a finite pre-Schwarzian norm.

\begin{thm}\label{p2-045}
Let $f=h\overline{g}$ be a sense-preserving logharmonic mapping with the dilatation $\omega$ given by \eqref{p2-010} and the analytic function $\psi$ be such that $\psi'=h'g$ in $\mathbb{D}.$ Then either, $||P_f||=||P_{\psi}||=\infty$ or, both  $||P_f||$ and $||P_{\psi}||$ are finite. If $||P_f||$ is finite, then
\begin{align}\label{p2-050}
\Big|||P_f||-||P_{\psi}||\Big|\leq 1,
\end{align}
and the constant $1$ is sharp.
\end{thm}

From Theorem \ref{p2-045}, one can find the estimate of the pre-Schwarzian norm of sense-preserving logharmonic mapping $f=h\overline{g}$ when the associated analytic function $\psi=h'g$ has certain geometry. We list a few which come immediately from the theorem.

\begin{cor}\label{p2-052}
Let $f=h\overline{g}$ be a sense-preserving logharmonic mapping with the dilatation $\omega$ given by \eqref{p2-010} and the analytic function $\psi$ be such that $\psi'=h'g$ in $\mathbb{D}.$
\begin{itemize}
\item[(i)] If $\psi$ is univalent, then $||P_f||\leq7.$
\item[(ii)] If $\psi$ is convex mapping, then $||P_f||\leq5.$
\end{itemize}
\end{cor}


Corollary \ref{p2-052} shows that the  univalence of $\psi$ provides a bound of the pre-Schwarzian norm $||P_f||$ of a logharmonic mapping $f.$ Now we wish to find a condition on $||P_f||$ under which the associated analytic function $\psi$ is univalent. Let us suppose that $||\omega^*||=\sup_{z\in\mathbb{D}}\omega^*(z),$ where $\omega^*$ is the  hyperbolic derivative of an analytic function $\omega:\mathbb{D}\rightarrow\mathbb{D}$ given by
\begin{align*}
\omega^*(z)=\frac{|\omega'(z)|(1-|z|^2)}{1-|\omega(z)|^2}.
\end{align*}
\begin{thm}\label{p2-055}
Let $f=h\overline{g}$ be a sense-preserving logharmonic mapping with the dilatation $\omega$ given by \eqref{p2-010} and the analytic function $\psi$ be such that $\psi'=h'g$ in $\mathbb{D}.$  If
\begin{align}\label{p2-060}
||P_f||+||\omega^*||\leq 1,
\end{align}
 then $\psi$ is univalent in $\mathbb{D}.$
\end{thm}

Properties of a sense-preserving logharmonic mapping of the form $f(z)=h(z)\overline{g(z)}$ have been studied in connection with the analytic function $zh(z)/g(z)$ in different context. Indeed, the analytic function $zh(z)/g(z)$ plays a vital role in shaping the geometry of $f(z).$  Many geometric aspects such as starlikeness, close-to-convexity, etc. of $f(z)$ is directly related to the function $zh(z)/g(z).$ For more detail, see the survey article \cite{Abdulhadi-Ali-2012} and bibliography therein. In this regard, we establish a connection between  a nonvanishing logharmonic mapping $f$ and its analytic part $h.$

\begin{thm}\label{p2-065}
Let $f(z)=h(z)\overline{g(z)}$ be a nonvanishing logharmonic mapping with dilatation $\omega(z).$  Then $f(z)$ is logharmonic Bloch if and only if $\log{h(z)}$ is analytic Bloch function.
\end{thm}

Next, we provide a necessary and sufficient condition for a nonvanishing logharmonic mapping to have a finite pre-Schwarzian norm. Note that Theorem \ref{p2-045} provides a similar condition for a logharmonic mapping that may vanish.
\begin{thm}\label{p2-070}
Let $f=h\overline{g}$ be a nonvanishing logharmonic Bloch mapping in the unit disk $\mathbb{D}$ with dilatation $\omega.$ Then  the pre-Schwarzian norm $||P_f||$ of $f$ is finite if and only if $||P_h||$ is finite. Moreover, if $||P_h||$ is finite then
\begin{align*}
\Big|||P_f||-||P_h||\Big|\leq \beta_{\log{g}}+1
\end{align*}
and the estimate is sharp.
\end{thm}

\begin{thm}\label{p2-090}
Let $f(z)=h(z)\overline{g(z)}\in\mathcal{B}_{Lh}$ be a nonvanishing logharmonic Bloch mapping with $h(0)=g(0)=1=h'(0).$ Then $\log{f}$ is an uniformly locally univalent harmonic mapping.
\end{thm}

In 1970, Pommerenke \cite{Pommerenke-1970} represented an analytic Bloch mapping $h$ in terms of the logarithm of some univalent analytic function. Indeed, He has shown that an analytic mapping $h$ is Bloch if and only if there exists a function $g(z)$ analytic and univalent in the unit disk and a constant $a > 0$
such that $h(z)=a\log{g'(z)}.$ Here, we present a similar type of representation for a nonvanishing logharmonic Bloch mapping.

\begin{thm}\label{p2-092}
A nonvanishing logharmonic mapping $f(z)=h(z)\overline{g(z)}$ is Bloch if and only if it is of the form $f(z)=H'(z)^{\lambda_1}\overline{G'(z)^{\lambda_2}},$ where $\lambda_1,\lambda_2>0$ are finite real numbers and $H, G\in\mathcal{S}_0.$
\end{thm}

Here, we note that $\lambda_1>\lambda_2$ whenever $H=G.$ Indeed, the logharmonic mapping $f=h\overline{g}=H'(z)^{\lambda_1}\overline{H'(z)^{\lambda_2}}$ has dilatation $\omega=\lambda_2/\lambda_1.$ We conclude this section with growth and distortion type theorems for a special type of logharmonic mappings $f=h\overline{g}$  along with for its analytic part $h$ and co-analytic part $g.$

\begin{thm}\label{p2-095}
Let $f(z)=h(z)\overline{g(z)}=H'(z)^{\lambda_1}\overline{G'(z)^{\lambda_2}}$ be a nonvanishing logharmonic mapping with $H, G\in\mathcal{S}_1$ and for some $\lambda_1,\lambda_2>0.$ Then for $|z|=r<1,$
\begin{itemize}
\item[(i)] $\left[\frac{1-r}{(1+r)^3}\right]^{\lambda_1}\leq |h(z)|\leq \left[\frac{1+r}{(1-r)^3}\right]^{\lambda_1}$ and  $\left[\frac{1-r}{(1+r)^3}\right]^{\lambda_2}\leq |g(z)|\leq \left[\frac{1+r}{(1-r)^3}\right]^{\lambda_2},$

\item[(ii)] $\left[\frac{1-r}{(1+r)^3}\right]^{\lambda_1+\lambda_2}\leq |f(z)|\leq \left[\frac{1+r}{(1-r)^3}\right]^{\lambda_1+\lambda_2},$

\item[(iii)] $\sup_{z\in\mathbb{D}}(1-|z|^2)\left|\frac{f_z(z)}{f(z)}\right|\leq 6\lambda_1.$

\end{itemize}
All the estimates are sharp for $\lambda_1>\lambda_2.$
\end{thm}


\section{Proof of Main Results}\label{Proof of Main Results}
In this section, we sequentially prove all the results stated in Section \ref{background And Main Results}.
\begin{proof}[\textbf{Proof of Theorem \ref{p2-045}}]
Since $f=h\overline{g}$ is a sense-preserving logharmonic mapping with the dilatation $\omega=g'h/h'g$, from \eqref{p2-040} the pre-Schwarzian derivative of $f$ is
\begin{align}\label{p2-110}
P_f=\frac{h''}{h'}+\frac{g'}{g}-\frac{\overline{\omega}\omega'}{1-|\omega|^2}.
\end{align}
As $\psi'(z)=h'(z)g(z)\neq0,$ it follows that the pre-Schwarzian derivative $P_{\psi}$ of $\psi$ is given by
\begin{align*}
P_{\psi}=\frac{\psi''}{\psi}=\frac{h''}{h'}+\frac{g'}{g}
\end{align*}
and so from \eqref{p2-110}, we have
\begin{align}\label{p2-130}
P_{\psi}-P_f=\frac{\overline{\omega}\omega'}{1-|\omega|^2}.
\end{align}
Therefore, by Schwarz-Pick lemma
\begin{align*}
(1-|z|^2)\Big||P_{\psi}|-|P_f|\Big|\leq & (1-|z|^2)\Big|P_{\psi}-P_f\Big| =  (1-|z|^2)\left|\frac{\overline{\omega}\omega'}{1-|\omega|^2}\right|\\
\leq & \frac{1-|z|^2}{1-|z|^2}=1.
\end{align*}
This shows that $||P_f||$ is finite if and only if $||P_{\psi}||$ is finite. Moreover, if $||P_f||<\infty$ then
\begin{align*}
\Big|||P_f||-||P_{\psi}||\Big|
& \leq \sup_{z\in\mathbb{D}} (1-|z|^2)\Big||P_f(z)|-|P_{\psi}(z)|\Big|\leq 1.
\end{align*}

To show that the constant $1$ is sharp, we consider the logharmonic mapping $f(z)=h(z)\overline{g(z)}$ with $h(z)=1/(1-z)$ and dilatation $\omega(z)=z.$ It  is quite easy to get $g(z)=e^{-z}/(1-z)$ and $\psi'(z)=e^{-z}/(1-z)^3.$ Thus we see that $f$ is nonvanishing. It is a simple exercise to see that the pre-Schwarzian derivative $P_f$ of $f$ is
\begin{align*}
P_f(z)=\frac{2+z}{1-z}-\frac{\overline{z}}{1-|z|^2}
\end{align*}
and the pre-Schwarzian derivative $P_{\psi}$ of the associated analytic function $\psi=h'g$ is
\begin{align*}
P_{\psi}(z)=\frac{2+z}{1-z}.
\end{align*}
Then it is easy to see that $||P_f||=5$ and $||P_{\psi}||=6$ which implies that
\begin{align*}
\Big|||P_{\psi}||-||P_f||\Big|=1.
\end{align*}

Additionally, the sharpness can also be seen from the logharmonic mapping $f(z)=h(z)\overline{g(z)}$ with $h(z)=z/(1-z)$ and dilatation $\omega(z)=z.$ A simple calculation gives $g(z)=1/(1-z)$ and $\psi'(z)=h'(z)g(z)=1/(1-z)^3.$ Note that $\psi(z)=z(2-z)/2(1-z)^2$ is univalent in $\mathbb{D}.$ In this case, a direct computation shows
\begin{align*}
P_f(z)=\frac{3}{1-z}-\frac{\overline{z}}{1-|z|^2}~~\quad\text{and}\quad P_{\psi}(z)=\frac{3}{1-z}
\end{align*} from which we get
\begin{align*}
||P_f||=5\quad\text{and}\quad ||P_{\psi}||=6.
\end{align*}
Therefore,
$$\Big|||P_{\psi}||-||P_f||\Big|=1.$$

\end{proof}

\begin{proof}[\textbf{Proof of Corollary \ref{p2-052}}]
It is well known that $||P_{\psi}||\leq6$ if $\psi$ is univalent and $||P_{\psi}||\leq4$ if $\psi$ is convex analytic in $\mathbb{D}.$ Therefore both the results follows immediately from Theorem \ref{p2-050}.
\end{proof}

\begin{proof}[\textbf{Proof of Theorem \ref{p2-055}}]
Let $f=h\overline{g}$ be a logharmonic mapping with the dilatation $\omega=g'h/{gh'}$ and the analytic function $\psi$ is such that $\psi'=h'g.$ Then  from \eqref{p2-130} and \eqref{p2-060}, it follows that
\begin{align*}
||P_{\psi}||\leq  & ||P_f||+\sup_{z\in\mathbb{D}}\frac{|\overline{\omega}||\omega'|}{1-|\omega|^2}(1-|z|^2)\\
\leq & ||P_f||+||\omega^*||\\
\leq & 1.
\end{align*}
Therefore, $\psi$ is univalent by Becker's criterion \cite{Becker-1972}.\\

To show that the constant $1$ is sharp, we consider the logharmonic mapping $f(z)=h(z)\overline{g(z)},$ with $g(z)=1.$ Then the associated analytic function $\psi$ is given by $\psi(z)=h(z),~z\in\mathbb{D}.$ Thus, the sharpness follows from the fact that the same is sharp in the analytic case.
\end{proof}

\begin{proof}[\textbf{Proof of Theorem \ref{p2-065}}]
Let $f(z)=h(z)\overline{g(z)}$ be a logharmonic mapping with dilatation $\omega=g'h/gh'$ and $\log h(z)$ be an analytic Bloch function. Then
\begin{align}\label{p2-150}
\sup_{z\in\mathbb{D}}(1-|z|^2)\left|\frac{h'(z)}{h(z)}\right|<\infty
\end{align}
and so,
\begin{align}\label{p2-160}
\beta_{\log{g}}=\sup_{z\in\mathbb{D}}(1-|z|^2)\left|\frac{g'(z)}{g(z)}\right|= & \sup_{z\in\mathbb{D}}(1-|z|^2)\left|\frac{\omega(z)h'(z)}{h(z)}\right|\\
\leq & \sup_{z\in\mathbb{D}}(1-|z|^2)\left|\frac{h'(z)}{h(z)}\right|\nonumber\\
< & \infty.\nonumber
\end{align}

Thus from \eqref{p2-150} and \eqref{p2-160}, it follows that
\begin{align*}
\beta_f= & \sup_{z\in\mathbb{D}}(1-|z|^2)\left\{\left|\frac{h'(z)}{h(z)}\right|+\left|\frac{g'(z)}{g(z)}\right|\right\}\\
\leq & \sup_{z\in\mathbb{D}}(1-|z|^2)\left|\frac{h'(z)}{h(z)}\right|+\sup_{z\in\mathbb{D}}(1-|z|^2)\left|\frac{g'(z)}{g(z)}\right|\\
<& \infty.
\end{align*}
Therefore, $f$ is logharmonic Bloch mapping by \eqref{p2-015}.\\

Conversely, suppose that $f$ is a logharmonic Bloch mapping. Then
\begin{align*}
\beta_f= & \sup_{z\in\mathbb{D}}(1-|z|^2)\left\{\left|\frac{h'(z)}{h(z)}\right|+\left|\frac{g'(z)}{g(z)}\right|\right\}\\
= & \sup_{z\in\mathbb{D}}\left\{\left|(1-|z|^2)\frac{h'(z)}{h(z)}\right|+\left|(1-|z|^2)\frac{g'(z)}{g(z)}\right|\right\} \\
< & \infty
\end{align*}
which gives
$$\sup_{z\in\mathbb{D}}(1-|z|^2)\left|\frac{h'(z)}{h(z)}\right|=\sup_{z\in\mathbb{D}}\left|(1-|z|^2)\frac{h'(z)}{h(z)}\right|<\infty.$$ Thus, the function $\log{h}$ is analytic Bloch.
\end{proof}

\begin{proof}[\textbf{Proof of Theorem \ref{p2-070}}]
Since $f=h\overline{g}$ is a logharmonic Bloch mapping, by Theorem \ref{p2-065}, the function $\log{h(z)}$ is an analytic Bloch. Therefore,
$$\sup_{z\in\mathbb{D}}(1-|z|^2)\left|\frac{h'(z)}{h(z)}\right|<\infty$$
and so from \eqref{p2-160}, we get $\beta_{\log{g}}<\infty$.
The pre-Schwarzian derivative of $f$ is
\begin{align*}
P_f(z)=P_h(z)+\frac{g'(z)}{g(z)}-\frac{\omega'(z)\overline{\omega(z)}}{1-|\omega(z)|^2},
\end{align*}
and so by Schwarz-Pick lemma, we get
\begin{align}\label{p2-175}
(1-|z|^2)\Big||P_f(z)|-|P_h(z)|\Big|&\leq(1-|z|^2)|P_f(z)-P_h(z)|\\
                                 &=(1-|z|^2)\left|\frac{g'(z)}{g(z)}-\frac{\overline{\omega(z)}\omega'(z)}{1-|\omega(z)|^2}\right|\nonumber\\
                                 &\leq\sup_{z\in\mathbb{D}}(1-|z|^2)\left|\frac{g'(z)}{g(z)}\right|+1\nonumber\\
                                 &= \beta_{\log{g}}+1 \nonumber
\end{align}
which shows that $||P_f||$ is finite if and only if $||P_h||$ is finite. This immediately gives
\begin{align}\label{p2-190}
\Big|||P_f||-||P_h||\Big|\leq \beta_{\log{g}}+1.
\end{align}

To show the estimate is sharp, we consider the logharmonic mapping $F(z)=H(z)\overline{G(z)}$ with $H(z)=e^{h(z)}$ and $h(z)=\int_{0}^z\frac{du}{1-u}$ and dilatation $\omega(z)=\frac{t-z}{1-tz},~~t\in(0,1).$ As
\begin{align*}
\beta_{\log{H}}=\sup_{z\in\mathbb{D}}(1-|z|^2)\left|\frac{H'(z)}{H(z)}\right|\leq \sup_{z\in\mathbb{D}}\frac{(1-|z|^2)}{1-|z|}=2,
\end{align*}
it follows that $\log{H}$ is analytic Bloch and hence, $F$ is logharmonic Bloch. It is easy to find that
$$P_H=\frac{2}{1-z}$$
and
$$\frac{G'(z)}{G(z)}=\omega(z)\frac{H'(z)}{H(z)}=\frac{t-z}{1-tz}\frac{1}{1-z}.$$
Therefore, $||P_H||=4$ and
$$\beta_{\log{G}}=\sup\limits_{z\in\mathbb{D}}(1-|z|^2)\left|\frac{G'(z)}{G(z)}\right|=\sup\limits_{z\in\mathbb{D}}(1-|z|^2)\left|\frac{t-z}{1-tz}\frac{1}{1-z}\right|=2.$$
Thus from \eqref{p2-190}, the pre-Schwarzian norm of $F$ satisfies
$||P_F||\leq ||P_H||+\beta_{\log{G}}+1\leq 7.$ \\

On the other hand, the pre-Schwarzian norm of $F$ is given by
\begin{align}\label{p2-220}
||P_F||= & \sup_{z\in\mathbb{D}}(1-|z|^2)|P_F(z)|\\
= & \sup_{z\in\mathbb{D}}(1-|z|^2)\left|P_h(z)+(1+\omega(z))h'(z)-\frac{\omega'(z)\overline{\omega(z)}}{1-|\omega(z)|^2}\right|\nonumber\\
=& \sup_{z\in\mathbb{D}}(1-|z|^2)\left|\frac{1}{1-z}+\frac{(1+t)(1-z)}{1-tz}\frac{1}{1-z}-\frac{\overline{z}-t}{(1-tz)(1-|z|^2)}\right|\nonumber.
\end{align}
Along the positive real axis,
\begin{align*}
N(t)= & \sup_{0\leq r<1}\left|1+r+\frac{(1+t)(1-r^2)}{1-tr}-\frac{(r-t)(1-r^2)}{(1-tr)(1-r^2)}\right|\\
= & \sup_{0\leq r<1}\left|1+r+\frac{(1+t)(1-r^2)-(r-t)}{1-tr}\right|\\
=&  \sup_{0\leq r<1}E(r)
\end{align*}
where
 $$E(r)=1+r+\frac{(1+t)(1-r^2)-(r-t)}{1-tr}.$$
Clearly
 $$E'(r)=\frac{(2t+1)\{r(tr-2)+t\}}{(1-tr)^2}.$$
It is an easy exercise to see that the roots of $E'(r)=0$ are $r_0=(1-\sqrt{1-t^2})/t$ and $r_1=(1+\sqrt{1-t^2})/t$ but $r_1$ does not lie in $(0,1).$ So the maximum value of $E(r)$ is attained at $r_0.$ Hence,
\begin{align}\label{p2-235}
N_t=E(r_0)=\frac{2 - 2\sqrt{1 - t^2}+t(4 + t - 4\sqrt{1 - t^2})}{t^2}.
\end{align}
From \eqref{p2-220} and \eqref{p2-235}, it follows that $N_t\leq||P_F||\leq 7.$ Also, it is simple to see that $N_t\rightarrow 7$ as $t\rightarrow 1^{-}.$ This shows that the inequality \eqref{p2-190} is sharp.
\end{proof}

The extremal function constructed in the previous theorem can be generalized, which is provided in the next example.

\begin{example}
Let $\mathcal{M}$ denote the family of sense-preserving logharmonic mappings $F=H(z)\overline{G(z)}$ with $H(z)=e^{h(z)}$ and $h(z)=\int_{0}^z\frac{d\zeta}{1-\epsilon(\zeta)}$ where $\epsilon:\mathbb{D}\rightarrow\mathbb{D}$ is an analytic function such that $\epsilon(0)=0$ and with the dilatation $\omega:\mathbb{D}\rightarrow\mathbb{D}$. It is easy to show that every $F\in\mathcal{M}$ is logharmonic Bloch and it satisfies the sharp estimate $||P_F||\leq7$.
\end{example}

A family $\mathcal{F}$ of sense-preserving harmonic mappings $F = H + \overline{G}$ in $\mathbb{D},$ normalized by $ H(0) = G(0) = 0$ and $H'(0) = 1$ is said to be linear invariant if it is closed under the  Koebe transform
\begin{align*}
   L_\phi(F)(z) = \frac{F(\phi(z))-F(\phi(0))}{F_z(\phi(0))\phi'(0)},\quad \phi(z)=(z-\alpha)/(1-\overline{\alpha}z),~\alpha\in\mathbb{D},
\end{align*}
and is said to be affine invariant if it is closed under the affine transform
\begin{align*}
A_s(F)(z) = \frac{F(z)-\overline{s F(z)}}{1-\overline{s}G'(0)}, \quad |s|<1.
\end{align*}
Now, we recall some known results required to prove Theorem \ref{p2-090}.

\begin{lem}\label{p2-240}\cite{Liu-Ponnusamy-2018}
If $f=h\overline{g}\in\mathcal{B}_{Lh},$ then
\begin{itemize}
\item[(i)] $f^a\overline{f}^b\in\mathcal{B}_{Lh}$ for any $a,b\in\mathbb{C}$ (affine invariance),
\item[(ii)] $f\circ\phi\in\mathcal{B}_{Lh}$ for any automorphism $\phi$ of $\mathbb{D}$ (linear invariance).
\end{itemize}
\end{lem}

\begin{lem}\label{p2-250} \cite{Anderson-Clunie-Pommerenke-1974}
Let $H(z)=\sum\limits_{n=1}^\infty a_nz^n$ be an analytic Bloch function. Then
\begin{align*}
|a_n|\leq 2||H||_\beta=2\sup_{z\in\mathbb{D}}(1-|z|^2)|H'(z)|~~\quad\text{for}~~ n\geq1.
\end{align*}
\end{lem}

\begin{proof}[\textbf{Proof of Theorem \ref{p2-090}}]
We know that the Bloch space $\mathcal{B}_{Lh}$ of  sense-preserving nonvanishing logharmonic Bloch mappings is linear and affine invariant. Let us consider the classes
$$\mathcal{J}=\{\log{f}: f=h\overline{g}\in\mathcal{B}_{Lh},~~h(0)=g(0)=h'(0)=1\}$$
 and
 $$\mathcal{J}^0=\{\log{f}: f=h\overline{g}\in\mathcal{B}_{Lh},~~h(0)=g(0)=h'(0)=1,~g'(0)=0\}.$$
 It is clear that every $F\in \mathcal{J}$ has the form $F(z)=H(z)+\overline{G(z)}$ with
\begin{align}\label{p2-260}
H(z)=z+\sum\limits_{n=2}^\infty a_n z^n\quad\text{and}\quad G(z)=\sum\limits_{n=1}^\infty b_nz^n.
\end{align}

We first show that $\mathcal{J}$ is a linear and affine invariant family of normalized harmonic mappings.\\

\textbf{Affine Invariance:} For $F=\log{f}\in \mathcal{J}$ with $f=h\overline{g}\in\mathcal{B}_{Lh}$ and a constant $s\in\mathbb{D},$ we have
\begin{align}\label{p2-270}
A_s[F](z)= & \frac{F(z)+s\overline{F(z)}}{1+s\overline{F_z(0)}}\\
          = & \frac{\log{f}+s\overline{\log{f}}}{1+sg'(0)}\nonumber\\
          = & \frac{\log{f(z){\overline{f(z)}}^s}}{a},\quad\text{where}~~ a=1+sg'(0)\nonumber\\
          = & \log{f^{1/a}(z)\overline{f(z)}^{s/a}}=\log{f_a(z)}\nonumber
\end{align}
where $f_a=f^{1/a}\overline{f}^{s/a}=h_a\overline{g_a}$ with $h_a=h^{1/a}g^{s/a}$ and $g_a=h^{\overline{s/a}}g^{\overline{1/a}}.$ From Lemma \ref{p2-240}, it follows that $f_a=f^{1/a}\overline{f}^{s/a}\in\mathcal{B}_{Lh}$ for $1/a,~s/a\in\mathbb{C}.$ As,

$$h_a'(z)=\frac{h^{1/a}g^{s/a}}{a}\left\{\frac{h'(z)}{h(z)}+\frac{sg'(z)}{g(z)}\right\},$$
from the conditions $h(0)=g(0)=h'(0)=1$ and $a=1+sg'(0),$ one can easily get $h_a(0)=g_a(0)=h_a'(0)=1$ and so the affine transform $A_s[F](z)\in \mathcal{J}$ for all $F\in \mathcal{J}$ and $|s|<1.$  Therefore the family $\mathcal{J}$ is affine invariant.\\

\textbf{Linear invariance:} For $F=\log{f}\in \mathcal{J}$ with $f=h\overline{g}\in\mathcal{B}_{Lh}$ and the automorphism $\phi(z)=(z-\alpha)/(1-\overline{\alpha}z),~~\alpha\in\mathbb{D},$ we have
\begin{align}\label{p2-280}
L_\phi[F](z)= & \frac{F(\phi(z))-F(\phi(0))}{F_z(\phi(0))\phi'(0)}\\
            = & \frac{\log{f(\phi(z))}-\log{f(\alpha)}}{b}\quad\text{where}~~b=F_z(\phi(0))=\frac{h'(\alpha)}{h(\alpha)}\phi'(0)\nonumber\\
            = & \log\left(\frac{f(\phi(z))}{f(\alpha)}\right)^{1/b}=\log{f_b(z)}\nonumber
\end{align}
where
$$f_b(z)=\left(\frac{f(\phi(z))}{f(\alpha)}\right)^{1/b}=h_b(z)\overline{g_b(z)}$$
with
$$ h_b(z)=\left(\frac{h(\phi(z))}{h(\alpha)}\right)^{1/b}\quad\text{and}\quad g_b(z)=\left(\frac{h(\phi(z))}{g(\alpha)}\right)^{\overline{1/b}}.$$

From Lemma \ref{p2-240}, $f_b(z)\in \mathcal{B}_{Lh}$ for $1/b , f(\alpha)\in \mathbb{C}$ as $\mathcal{B}_{Lh}$ is a linear invariant family. Since, $$h_b'(z)=\frac{1}{b}\left(\frac{h(\phi(z))}{h(\alpha)}\right)^{1/b}\frac{h'(\phi(z))}{h(\alpha)}\phi'(z),$$
from $h(0)=g(0)=1$ and $b=\phi'(0) h'(\alpha)/h(\alpha)$, it follows that $h_b(0)=g_b(0)=1=h_b'(0)$ and so $L_\phi[F]$ belong to $\mathcal{J}.$ Hence, $\mathcal{J}$ is a linear invariant family of normalized harmonic mappings.\\

Let $F(z)=H(z)+\overline{G(z)}$ be a member of the class $\mathcal{J}.$ Then by Theorem \ref{p2-065}, the analytic function $H(z)=\log{h}$ is an analytic Bloch function and so
\begin{align}\label{p2-290}
||H||_\mathcal{B}= & |H(0)|+\sup_{z\in\mathbb{D}}(1-|z|^2)|H'(z)|<\infty.
\end{align}
Since the family $\mathcal{J}$ is affine and linear invariant, the pre-Schwarzian norm estimate of $F$ is given by (see \cite{Graf-2016} )
\begin{align*}
||P_F|| &= \sup_{z\in\mathbb{D}}(1-|z|^2)|P_F(z)|\\
        & \leq 2+2\alpha_0 \quad\text{where}~~\alpha_0=\sup_{F\in \mathcal{J}^0}|a_2|.
  \end{align*}
From Lemma \ref{p2-250} and \eqref{p2-290}, it follows that
 \begin{align*}
   ||P_F|| \leq 2+2\alpha_0\leq 2+2\sup_{F\in \mathcal{J}^0}|a_2|\leq 2+4||H||_{\beta} < \infty.
\end{align*}
We see that the pre-Schwarzian norm of $F=\log{f}$ is finite and hence $\log{f}$ is  uniformly locally univalent.
\end{proof}

Before proving our next theorem, we recall a result by Pommerenke \cite{Pommerenke-1970}, which helps us understand the proof.

\begin{lem}\label{p2-300}
An analytic function $h$ is Bloch if and only if there exists an analytic function $g\in\mathcal{S}_0$ and a finite positive constant $\lambda$ such that $h(z)=\lambda\log{g'(z)}.$
\end{lem}

\begin{proof}[\textbf{Proof of Theorem \ref{p2-092}}]
Suppose that $f=h\overline{g}$ is logharmonic Bloch mapping. Then from Theorem \ref{p2-065} the function $\log{h(z)}$
 is analytic Bloch and so is $\log{g(z)}.$ Therefore, by Lemma \ref{p2-300}, there exist two finite positive real constants $\lambda_1, \lambda_2$ and two analytic functions $H, G\in\mathcal{S}_0$ such that
 \begin{align*}
 h(z)=H'(z)^{\lambda_1}\quad\text{and}\quad g(z)=G'(z)^{\lambda_2},
 \end{align*}
 which further implies that $$f(z)=H'(z)^{\lambda_1}\overline{G'(z)^{\lambda_2}}.$$\\

 Conversely, suppose $f(z)=h(z)\overline{g(z)}=H'(z)^{\lambda_1}\overline{G'(z)^{\lambda_2}}$ with finite positive numbers $\lambda_1, \lambda_2.$ Then
 \begin{align*}
 \beta_{\log{h}}=\sup_{z\in\mathbb{D}}(1-|z|^2)^2\big|\frac{h'(z)}{h(z)}\big|=\lambda_1||P_H||\leq6\lambda_1<\infty,
 \end{align*}
 which is sufficient to show that $\log{h}$ is analytic Bloch and hence, by Theorem \ref{p2-065}, $f$ is a logharmonic Bloch mapping.
 \end{proof}

\begin{proof}[\textbf{Proof of Theorem \ref{p2-095}}]
Let $f=h\overline{g}$ be a logharmonic Bloch mapping of the form $f(z)=H'(z)^{\lambda_1}\overline{G'(z)^{\lambda_2}}$ with $H, G\in\mathcal{S}_1.$ The  distortion theorem for the analytic functions $H$ and $G$ (see \cite{Duren-1983}), for $|z|=r<1$ is
\begin{align}\label{p2-330}
\frac{1-r}{(1+r)^3}\leq |H'(z)|,|G'(z)|\leq \frac{1+r}{(1-r)^3}.
\end{align}
It is very easy to see that $$|h(z)|=|H'(z)|^{\lambda_1},~|g(z)|=|G'(z)|^{\lambda_2},~|f(z)|=|H'(z)|^{\lambda_1}|G'(z)|^{\lambda_2}.$$
Thus $(i)$ and $(ii)$ follow immediately from \eqref{p2-330}.\\

 Since $H\in\mathcal{S}_1,$ it follows that the pre-Schwarzian norm of $H$ satisfies
\begin{align}\label{p2-340}
||P_H||\leq 6.
\end{align}
A direct calculation shows that
$$\frac{f_z(z)}{f(z)}=\lambda_1\frac{H''(z)}{H'(z)}=\lambda_1P_H(z)$$ and so from \eqref{p2-340}
\begin{align*}
\sup_{z\in\mathbb{D}}(1-|z|^2)\left|\frac{f_z(z)}{f(z)}\right|= \sup_{z\in\mathbb{D}}(1-|z|^2)\lambda_1|P_H(z)|=\lambda_1||P_H||\leq 6\lambda_1.
\end{align*}

To show that the estimates are sharp, we consider the sense-preserving logharmonic mapping $f(z)=k'(z)^{\lambda_1}\overline{k'(z)^{\lambda_2}}$ with $k(z)=z/(1-z)^2$. Here $k'(z)=(1+z)/(1-z)^3.$ Thus for $z=r\in[0,1),$ we have
$$h(r)=\left[\frac{1+r}{(1-r)^3}\right]^{\lambda_1},~g(r)=\left[\frac{1+r}{(1-r)^3}\right]^{\lambda_2},~f(r)=\left[\frac{1+r}{(1-r)^3}\right]^{\lambda_1}\left[\frac{1+r}{(1-r)^3}\right]^{\lambda_2}.$$
Therefore the right hand side inequalities are sharp at $z=r$ and the left hand side inequalities are at $z=-r.$\\

 It is very easy to see that $f_z/f=\lambda_1P_k$ and so
\begin{align*}
\sup_{z\in\mathbb{D}}(1-|z|^2)\left|\frac{f_z(z)}{f(z)}\right|=\sup_{z\in\mathbb{D}}(1-|z|^2)\lambda_1|P_k(z)|= \lambda_1||P_k||= 6\lambda_1.
\end{align*}
\end{proof}

\noindent\textbf{Declarations}\\

\noindent\textbf{Conflict of interest:} The authors declare that they have no conflict of interest.\\

\noindent\textbf{Data availability:}
Data sharing not applicable to this article as no data sets were generated or analyzed during the current study.\\

\noindent\textbf{Authors Contributions:}
All authors contributed equally to the investigation of the problem and the order of the authors is given alphabetically according to their surname. All authors read and approved the final manuscript. \\

\noindent\textbf{Acknowledgement:}
The second named author thanks the Department Of Science and Technology, Ministry Of Science and Technology, Government Of India
for the financial support through DST-INSPIRE Fellowship (No. DST/INSPIRE Fellowship/2018/IF180967).\\


\begin{thebibliography}{99}


\bibitem{Abdulhadi-1988}
{\sc Z. Abdulhadi, D. Bshouty}, Univalent functions in $H\overline{H}$, {\it Tran. Amer. Math. Soc.}, {\bf 305}(2) (1988), 841-849.

\bibitem{Abdulhadi-Ali-2012}
{\sc Z. Abdulhadi, R.M. Ali}, Univalent logharmonic mappings in the plane, {\it Abstr. Appl. Anal.}, {\bf 2012} (2012), Art. ID 721943, 32 pp.

\bibitem{Anderson-Clunie-Pommerenke-1974}
{\sc J. Anderson, J. Clunie, Ch. Pommerenke}, On Bloch functions and normal functions, Walter de Gruyter, Berlin/New York Berlin, New York, (1974), 12-37.

\bibitem{Becker-1972}
{\sc J. Becker}, Löwnersche Differentialgleichung und quasikonform fortsetzbare schlichte Funktionen, {\it J. Reine Angew. Math.}, {\bf 255} (1972), 23-43.

\bibitem{Becker-Pommerenke-1984}
{\sc Becker, Pommerenke}, Schlichtheitskriterien und Jordangebiete, {\it J. Reine Angew. Math.}, {\bf 354} (1984), 74-94.

\bibitem{Bravo-Hernandez-Ponnusamy-Venegas-2022}
{\sc V. Bravo, R. Hern{\'a}ndez, S. Ponnusamy, O. Venegas}, Pre-Schwarzian and Schwarzian derivatives of Logharmonic mapppings, {\it Monatsh. Math.}, (2022), 1-22.

\bibitem{Chuaqui-Duren-Osgood-2003}
{\sc M. Chuaqui, P. Duren, B. Osgood}, The Schwarzian derivative for harmonic mappings, {\it J. Anal. Math.}, {\bf 91}(1) (2003), 329-351.

\bibitem{Duren-1983}
P. L. Duren, Univalent Functions, Springer-Verlag (1983).

 \bibitem{Graf-2016}
{\sc S. Yu. Graf}, On the schwarzian norm of harmonic mappings, {\it Probl. Anal. Issues Anal.}, {\bf 5}(2) (2016), 20-32.

\bibitem{Hernandez-Martin-2013}
 {\sc R. Hern{\'a}ndez, M. J. Mart{\'\i}n}, Quasi-conformal extensions of harmonic mappings in the plane, {\it Ann. Acad. Sci. Fenn. Ser. A. I Math.}, {\bf 38} (2013), 617-630.

\bibitem{Hernandez-Martin-2015}
{\sc R. Hern{\'a}ndez, M. J. Mart{\'\i}n},  Pre-Schwarzian and Schwarzian derivatives of harmonic mappings, {\it J. Geomet. Anal.}, {\bf 25}(1)  (2015), 64-91.

\bibitem{Huusko-Martin-2017}
{\ J. M. Huusko, M. J. Mart{\'\i}n}, Criteria for bounded valence of harmonic mappings, {\it Comput.Methods Funct. Theory}, {\bf 17}(4) (2017), 603-612.

\bibitem{Kruas-1932}
{\sc W. Kraus}, Uber den Zusammenhang eigner Characterstiken eines einfach zusammenhangenden Bereiches mit der Kreisabbildung, {\it Mitt. Math. Sem. Giessen}, {\bf 21}  (1932),  1-28.

\bibitem{Liu-Ponnusamy-2018}
{\sc Z. Liu, S. Ponnusamy}, Some properties of univalent log-harmonic mappings, {\it Filomat}, {\bf 32}(15) (2018), 5275-5288.

\bibitem{Liu-Ponnusamy--2018}
{\sc G. Liu, S. Ponnusamy}, Uniformly locally univalent harmonic mappings associated with the pre-Schwarzian norm, {\it Indagationes Mathematicae} {\bf 29}(2) (2018), 752-778.

\bibitem{Mao-Ponnusamy-Wang-2013}
{\sc Z Mao, S. Ponnusamy, X. Wang}, Schwarzian derivative and Landau’s theorem for logharmonic mappings, {\it Complex Var. Elliptic Equ.}, {\bf 58}(8) (2013), 1093-1107.

\bibitem{Nehari-1949}
{\sc Z. Nehari}, The Schwarzian derivative and schlicht functions, {\it Bull. Amer. Math. Soc.} {\bf 55}(6) (1949), 545-551.

\bibitem{Pommerenke-1970}
{\sc Ch. Pommerenke}, On Bloch functions, {\it J. London Math. Soc.}, {\bf 2}(2) (1970), 689-695.

\bibitem{Yamashita-1976}
{\sc S. Yamashita}, Almost locally univalent functions, \textit{Monatsh. Math.}, {\bf 81} (1976), 235-240.


%
%
%
%
%
%

%


%


%
%

%
%
%





\end{thebibliography}
\end{document}